\documentclass{article}
\usepackage{amsmath}
\usepackage{amsthm}
\usepackage{amsfonts}
\usepackage{latexsym}

\newcounter{alphthm}
\newtheorem{thm}{Theorem}[section]
\newtheorem{lem}[thm]{Lemma}
\newtheorem{cor}{Corollary}[section]

\theoremstyle{definition}

\newtheorem{rem}{Remark}[section]

\newcommand{\be}{\begin{equation}}
\newcommand{\ee}{\end{equation}}
\newcommand{\ben}{\begin{enumerate}}
\newcommand{\een}{\end{enumerate}}
\newcommand{\pa}{{\partial}}

\newcommand{\pxi}{{\pa \over \pa x^i}}

\vspace{7cm} \textwidth 14cm \textheight 21.6cm
\title{On Generalized Douglas-Weyl $(\alpha, \beta)$-Metrics\footnote{Acta Mathematica Sinica, English Series, Vol. 31, No. 10, (2015), 1611-1620}}
\author{A. Tayebi and H. Sadeghi}
\begin{document}

\maketitle
\begin{abstract}
In this paper, we  study generalized Douglas-Weyl  $(\alpha, \beta)$-metrics. Suppose that an regular $(\alpha, \beta)$-metric $F$ is not of Randers type. We prove that $F$ is a generalized Douglas-Weyl metric with vanishing S-curvature if and only if it is a Berwald metric. Moreover by ignoring the regularity, if $F$ is not a Berwald metric then we find a family of almost regular Finsler metrics which is not Douglas nor Weyl. As its application, we show that   generalized Douglas-Weyl square metric or Matsumoto metric  with isotropic mean Berwald curvature are Berwald metrics.
\end{abstract}

%It results that, an almost regular generalized Douglas-Weyl $(\alpha,\beta)$-metric  with vanishing $S$-curvature which is not Berwaldian reduces to the unicorn metric constructed by Shen.

\section{Introduction}\label{Int}
In \cite{Cheng}, Cheng proved that every non-Randers type $(\alpha, \beta)$-metric  of scalar
flag curvature   with vanishing S-curvature is a Berwald metric. Finsler metrics of scalar flag curvature are called Weyl metrics.  In \cite{Sak}, Sakaguchi showed that every Weyl metric is a generalized Douglas-Weyl metric. This motivates  us to study generalized Douglas-Weyl  $(\alpha, \beta)$-metric with vanishing S-curvature.

For considering generalized Douglas-Weyl metric, let us introduce some Finslerian notions.  Let $(M, F)$ be a Finsler manifold. In local coordinates, a curve $c(t)$ is a geodesic if and only if its coordinates $(c^i(t))$ satisfy $ \ddot c^i+2G^i(\dot c)=0$, where the local functions $G^i=G^i(x, y)$ are called the  spray coefficients.  $F$ is  called a Berwald metric, if  $G^i$  are quadratic in $y\in T_xM$  for any $x\in M$ or equivalently   $2G^i=\Gamma^i_{jk}(x)y^jy^k$. As a generalization of Berwald curvature,  B\'{a}cs\'{o}-Matsumoto introduced the notion of  Douglas metrics which are  projective invariants in Finsler geometry \cite{BM}.   $F$ is called a Douglas metric if $G^i=\frac{1}{2}\Gamma^i_{jk}(x)y^jy^k + P(x,y)y^i$.

Other than Douglas metrics, there is another projective invariant in Finsler geometry, namely
\[
D^i_{\ jkl|m}y^m= T_{jkl}y^i
\]
that is hold for some tensor $T_{jkl}$, where $D^i_{\ jkl|m}$ denotes the horizontal covariant derivatives of Douglas curvature $D^i_{\ jkl}$ with respect to the Berwald connection of $F$. For a manifold $M$, let ${\mathcal GDW}(M)$ denotes the class of all Finsler metrics satisfying in above relation for some tensor $T_{jkl}$. In \cite{BP}, B\'{a}cs\'{o}-Papp showed that ${\mathcal GDW}(M)$ is closed under projective changes. Then, Najafi-Shen-Tayebi  characterized  generalized Douglas-Weyl Randers metrics  \cite{NST1}. Recently,  the  authors with Peyghan prove that all generalized Douglas-Weyl spaces with vanishing Landsberg curvature  have vanishing the quantity $\bf H$ \cite{TSP}.

In order to find explicit examples of generalized Douglas-Weyl metrics, we consider $(\alpha, \beta)$-metrics.  This class of metrics was first introduced by Matsumoto which appear iteratively in formulating Physics, Mechanics,   Biology and Ecology, etc  \cite{BCZ}\cite{Mat}. An $(\alpha, \beta)$-metric is a Finsler metric of the form $F:=\alpha \phi(s)$, $s=\beta/\alpha$, where  $\phi=\phi(s)$  is a $C^\infty$ on $(-b_0, b_0)$, $\alpha=\sqrt{a_{ij}(x)y^iy^j}$ is a Riemannian metric and $\beta =b_i(x)y^i$ is a 1-form on  $M$. In particular,  $\phi= c_1\sqrt{1+c_2s^2}+c_3s$, $s=\beta/\alpha$,  are called Randers type metrics where  $c_1>0$, $c_2$ and $c_3$ are constant.

The notion of  S-curvature is originally introduced by Shen for the volume comparison theorem which interacts with other Riemannian and non-Riemannian curvatures \cite{Shvol}\cite{SX}\cite{TR}.   The Finsler metric $F$ is said to be of isotropic S-curvature if ${\bf S}=(n+1)cF$, where $c=c(x)$ is a scalar function on $M$. If c is a constant, then F is said to be of constant S-curvature. The  studies  how that the S-curvature plays a very important role in Finsler geometry \cite{MY}. Recently, Shen proved that every negatively curved Finsler metric with constant S-curvature on a closed manifold must be Riemannian  \cite{ShZ}.   In \cite{TR}, it is proved that every isotropic Berwald metric has isotropic S-curvature. The Finsler metrics with vanishing  S-curvature are of some important geometric structures which deserve to be studied deeply. For example, Shen proved that the Bishop-Gromov volume comparison holds for Finsler manifolds with vanishing S-curvature \cite{Sh}.

In this paper, we  characterize the generalized Douglas-Weyl $(\alpha,\beta)$-metrics with vanishing S-curvature. More precisely, we prove the following.

\begin{thm}\label{mainthm1}
Let $F=\alpha\phi(s)$, $s=\beta/\alpha$, be a regular  $(\alpha,\beta)$-metric on a manifold $M$. Suppose that $F$ is not a Finsler metric of Randers type. Then  $F$  is a generalized Douglas-Weyl metric with vanishing S-curvature if and only if it is a Berwald metric.
\end{thm}
According to Cheng's theorem  ( Theorem 4 in \cite{Cheng}), every non-Randers type $(\alpha, \beta)$-metric  on a manifold $M$ of dimension $n\geq 3$ is Weyl metric with vanishing S-curvature if and only if it is a Berwald metric with vanishing flag curvature. Here, we weaken Cheng' condition on the Weyl metrics to the generalized Douglas-Weyl metrics. We also delete the dimension's condition on manifold. Then Theorem \ref{mainthm1}, is an extension of Cheng's theorem.

We must mention that Theorem \ref{mainthm1} does not hold  for Finsler metrics of Randers
type. The family of  Randers metrics on $S^3$ constructed by Bao-Shen are generalized Douglas-Weyl metric with ${\bf S}=0$  which are not Berwaldian (see \cite{Sh}, page 164).

In Theorem \ref{mainthm1},  vanishing of  S-curvature is necessary. For example, consider following  Finsler metric on the unit ball $\mathbb B^n$
\begin{eqnarray*}\label{berwald}
F:=\frac{(\sqrt{(1-|x|^2)|y|^2+\langle x,y\rangle^2}+\langle
x,y\rangle)^2}{(1-|x|^2)^2\sqrt{(1-|x|^2)|y|^2+\langle
x,y\rangle^2}},
\end{eqnarray*}
where $|.|$ and $<,>$ denote the Euclidean norm and inner product in $\mathbb{R}^n$, respectively. This metric is projectively flat with flag curvature ${\bf K}=0$ (see page 96 in \cite{BCZ}). Therefore  $F$ is a Weyl metric and then it is a generalized Douglas-Weyl metric, also. $F$ satisfies  ${\bf S}\neq 0$ and  is not a Berwald metric.

Theorem \ref{mainthm1}, may not be hold for an $(\alpha,\beta)$-metric of  constant S-curvature. For example, at a point $x=(x^1, x^2)\in R^2$ and in the direction $y=(y^1, y^2)\in T_xR^2$, consider Riemannian metric $\alpha(x,y)=\sqrt{(y^1)^2+e^{2x^1}(y^2)^2}$ and one form $\beta(x, y):=y^1$. Then $s_{ij}=0$, $r_{ij}=a_{ij}-b_ib_j$ and $\epsilon=b=1$. Thus if $\phi=\phi(s)$ satisfies (\ref{P}) for some constant $k$, then $F=\alpha\phi(\beta/\alpha)$ has constant S-curvature  ${\bf S}=3kF$. Since every two-dimensional metric is a Weyl metric, then $F$ is a generalized Douglas-Weyl metric while it is not a Berwald metric.

Let $\phi = \phi(s)$ satisfy  $\phi(s) >0$ and $\phi(s) -s \phi'(s) + (b^2-s^2) \phi''(s) >0$, where $( |s|\leq b \leq b_o)$. A function $F=\alpha\phi(s)$ is called an almost regular $(\alpha,\beta)$-metric if $\beta$ satisfies that $||\beta_x||_{\alpha}\leq b_0$, $\forall x\in M$ \cite{ShC}. An almost regular $(\alpha,\beta)$-metric $F=\alpha\phi(s)$  might be singular (even not defined) in the two extremal directions $y\in  T_xM$ with $\beta(x, y) =\pm b_0\alpha(x, y)$. By assumptions of  Theorem \ref{mainthm1},   we find that  if $F$ is almost regular and not Berwaldian then it reduces to
\begin{eqnarray*}
\phi=c\exp\Bigg[\int_0^s{\frac{kt+q\sqrt{b^2-t^2}}{1+kt^2+qt\sqrt{b^2-t^2}}dt}\Bigg],
\end{eqnarray*}
where $c>0$, $q>0$ and  $k$  are real constants. The above Finsler metric  is not a  Douglas metric nor Weyl metric. It is not Landsberg metric, also (see Remark \ref{rem}).

\bigskip

Taking twice vertical covariant derivatives  of  $S$-curvature ${\bf S}$ gives rise $E$-curvature  ${\bf E}$.  The Finsler metric $F$ is  said to have isotropic mean Berwald curvature if ${\bf E}=\frac{n+1}{2}cF{\bf h}$, where $c=c(x)$ is a scalar function on $M$ and ${\bf h}=h_{ij}dx^i\otimes dx^j$ is the angular metric.  Among the $(\alpha, \beta)$-metrics, the  square metric $F=\alpha+2\beta+\beta^2/\alpha$  and the  Matsumoto metric $F=\alpha^2/(\alpha-\beta)$ are significant metric which constitute a majority of actual research.
\begin{cor}\label{cor1}
Let $F$ be a generalized Douglas-Weyl square metric or Matsumoto metric. Suppose that  $F$  has isotropic mean Berwald curvature. Then it reduces to a  Berwald metric.
\end{cor}

In \cite{CL}, Cheng-Lu proved that  every weakly Berwald square metric or Matsumoto metric of scalar flag curvature is a Berwald metric. Then Corollary \ref{cor1} is an extension of Cheng-Lu's theorem.

%-----------------------------------------------------------------------------------------------------------------------------------
\section{Preliminary}
%-----------------------------------------------------------------------------------------------------------------------------------
Given a Finsler manifold $(M,F)$, then a global vector field ${\bf G}$ is induced by $F$ on $TM_0:=TM-\{0\}$, which in a standard coordinate $(x^i,y^i)$ for $TM_0$ is given by ${\bf G}=y^i {{\partial} \over {\partial x^i}}-2G^i(x,y){{\partial} \over {\partial y^i}}$, where
\[
G^i:=\frac{1}{4}g^{il}\Big[\frac{\partial^2(F^2)}{\partial x^k \partial y^l}y^k-\frac{\partial(F^2)}{\partial x^l}\Big],\ \ \ \  y\in T_xM.
\]
The ${\bf G}$ is called the  spray associated  to $F$.  For  $y \in T_xM_0$, define ${\bf B}_y:T_xM\otimes T_xM \otimes T_xM\rightarrow T_xM$, ${\bf E}_y:T_xM \otimes T_xM\rightarrow \mathbb{R}$ and  ${\bf D}_y:T_xM\otimes T_xM \otimes T_xM\rightarrow T_xM$  by ${\bf B}_y(u, v, w):=B^i_{\ jkl}(y)u^jv^kw^l{{\partial } \over {\partial x^i}}|_x$, ${\bf E}_y(u,v):=E_{jk}(y)u^jv^k$ and ${\bf D}_y(u,v,w):=D^i_{\ jkl}(y)u^iv^jw^k\frac{\partial}{\partial x^i}|_{x}$, where
\[
B^i_{\ jkl}:={{\partial^3 G^i} \over {\partial y^j \partial y^k \partial y^l}},\ \ \ E_{jk}:={{1}\over{2}}B^m_{\ jkm},\ \ \ D^i_{\ jkl}:=B^i_{\ jkl}-\frac{\partial^3}{\partial y^j\partial y^k\partial y^l}\Bigg[\frac{2}{n+1}\frac{\partial G^m}{\partial y^m} y^i\Bigg].
\]
The $\bf B$, $\bf E$ and $\bf D$ are called the Berwald curvature, mean Berwald curvature and Douglas curvature, respectively.  Then $F$ is called a Berwald metric, weakly Berwald metric and Douglas metric if $\bf{B}=0$, $\bf{E}=0$ and $\bf{D}=0$, respectively. If ${\bf E}=0$, then $\bf{D}=\bf{B}$.

\bigskip
The notion of Riemann curvature for Riemann metrics can be extended
to Finsler metrics. For a non-zero vector $y \in T_xM_{0}$, the Riemann curvature is a family of linear transformation $\textbf{R}_y: T_xM \rightarrow T_xM$ with homogeneity ${\bf R}_{\lambda y}=\lambda^2 {\bf R}_y$, $\forall \lambda>0$ which is defined by
$\textbf{R}_y(u):=R^i_{\ k}(y)u^k {\partial \over {\partial x^i}}$, where
\[
R^i_{\ k}(y)=2{\partial G^i \over {\partial x^k}}-{\partial^2 G^i \over
{{\partial x^j}{\partial y^k}}}y^j+2G^j{\partial^2 G^i \over
{{\partial y^j}{\partial y^k}}}-{\partial G^i \over {\partial
y^j}}{\partial G^j \over {\partial y^k}}.
\]
The family $\textbf{R}:=\{\textbf{R}_y\}_{y\in TM_0}$ is called the Riemann curvature. A Finsler metric $F$ is said to be R-quadratic if ${\bf R}_y$ is quadratic in $y\in T_xM$ at each point $x\in M$.

\bigskip

Let $F:=\alpha\phi(s)$, $s=\beta/\alpha$, be an   $(\alpha, \beta)$-metric on a manifold $M$, where $\phi=\phi(s)$ is a $C^\infty$ function on the $(-b_0, b_0)$ with certain regularity, $\alpha=\sqrt{a_{ij}y^iy^j}$ is a Riemannian metric and $\beta=b_i(x)y^i$ is a 1-form on $M$.   Define $b_{i|j}$ by
\[
b_{i|j}\theta^j:=db_i-b_j\theta^j_i,
\]
where $\theta^i:=dx^i$ and $\theta^j_i:=\Gamma^j_{ik}dx^k$ denote
the Levi-Civita connection form of $\alpha$. Let
\begin{eqnarray*}
r_{ij}:=\frac{1}{2}(b_{i|j}+b_{j|i}), \ \ \ s_{ij}:=\frac{1}{2}(b_{i|j}-b_{j|i}).
\end{eqnarray*}
Clearly, $\beta$ is closed if and only if $s_{ij}=0$.  Put
\begin{eqnarray*}
&&r_{i0}: = r_{ij}y^j, \  \ r_{00}:=r_{ij}y^iy^j, \ \ r_j := b^i r_{ij},\ \ r_0:= r_j y^j,\\
&& \quad s_{i0}:= s_{ij}y^j, \  \  s_j:=b^i s_{ij}, \ \  s_0 := s_j y^j.
\end{eqnarray*}
Now, let $G^i=G^i(x,y)$ and $\bar{G}^i_{\alpha}=\bar{G}^i_{\alpha}(x,y)$ denote the
coefficients  of $F$ and  $\alpha$ respectively in the same coordinate system. By definition, we have
\begin{eqnarray}
G^i=G^i_{\alpha}+\alpha Q s^i_0+(-2Q\alpha s_0+r_{00})(\Theta\frac{y^i}{\alpha}+\Psi b^i),\label{G1}
\end{eqnarray}
where
\begin{eqnarray*}
Q:=\frac{\phi'}{\phi-s\phi'}, \ \ \Theta:={\phi\phi' -s (\phi\phi'' +\phi'\phi')\over 2 \phi \Big[  ( \phi -s\phi')+( b^2 -s^2)\phi''  \Big ]}, \ \ \ \Psi:={1\over 2} { \phi'' \over  ( \phi -s \phi')+( b^2 -s^2)\phi''}.
\end{eqnarray*}

\bigskip

For a Finsler metric $F$ on an $n$-dimensional manifold $M$, the
Busemann-Hausdorff volume form $dV_F = \sigma_F(x) dx^1 \cdots
dx^n$ is defined by
$$
\sigma_F(x) := {{\rm Vol} (\Bbb B^n(1))
\over {\rm Vol} \Big [ (y^i)\in R^n \ \Big | \ F \Big ( y^i
\pxi|_x \Big ) < 1 \Big ] } .\label{dV}
$$
Let $G^i$ denote the geodesic coefficients of $F$ in the same
local coordinate system. The S-curvature is defined by
$$
 {\bf S}({\bf y}) := {\pa G^i\over \pa y^i}(x,y) - y^i {\pa \over \pa x^i}
\Big [ \ln \sigma_F (x)\Big ],\label{Slocal}
$$
where ${\bf y}=y^i\pxi|_x\in T_xM$.  ${\bf S}$ said to be   isotropic if there is a scalar functions $c=c(x)$ on $M$ such that ${\bf S}=(n+1)c F$.

\bigskip

Now, let $\phi=\phi(s)$ be a positive $C^{\infty}$ function on
$(-b_0,b_0)$. For a number $b\in[0,b_0)$, put
\begin{eqnarray*}
&&\Phi:=-(Q-sQ')[n\Delta+1+sQ]-(b^2-s^2)(1+sQ)Q''
\end{eqnarray*}
where $\Delta:=1+sQ+(b^2-s^2)Q'$.  In \cite{ChSh3}, Cheng-Shen  characterize $(\alpha,\beta)$-metrics with isotropic
S-curvature.
\begin{lem}\label{CS0}{\rm (\cite{ChSh3})}
\emph{Let $F=\alpha\phi(s)$, $s=\frac{\beta}{\alpha}$, be an $(\alpha,\beta)$-metric on a $n$-dimensional manifold $M$.
 Suppose that $\phi\neq c_1\sqrt{1+c_2s^2}+c_3s$ for any constant $c_1>0$, $c_2$ and
$c_3$. Then $F$ is of isotropic S-curvature if and only if one of the following holds\\
(a) $\beta$ satisfies
\begin{eqnarray}
r_{ij}=\varepsilon(b^2a_{ij}-b_ib_j), \ \ \ s_j=0,\label{45}
\end{eqnarray}
where $\varepsilon=\varepsilon(x)$ is a scalar function, $b:=\|\beta_x\|_{\alpha}$ and
$\phi=\phi(s)$ satisfies
\begin{eqnarray}
\Phi=-2(n+1)k\frac{\phi\Delta^2}{b^2-s^2},\label{P}
\end{eqnarray}
where $k$ is a constant. In this case, $\textbf{S}=(n+1)cF$ with $c=k\varepsilon$.\\
(b) $\beta$ satisfies
\begin{eqnarray}
r_{ij}=0,\hspace{.5cm}s_j=0\label{44}
\end{eqnarray}
In this case, $\textbf{S}=0$.}
\end{lem}

%---------------------------------------------------------------------------------------------------------------
\section{Sakaguchi Theorem}
%---------------------------------------------------------------------------------------------------------------
Here, we give a proof of Sakaguch's Theorem. Our approach  is completely  different from Sakaguchi'method. For this aim, we  remark the following.
\begin{lem}{\rm (\cite{Sh})}
\emph{The following  Bianchi identities  holds:}
\begin{eqnarray}
&&R^i_{\ jkl|m}+ R^i_{\ jlm|k}+R^i_{\ jmk|l}=B^i_{\ jku}R^u_{\ lm}+B^i_{\ jlu}R^u_{\ km}+B^i_{\ klu}R^u_{\ jm},\\
&&B^i_{\ jkl|m}- B^i_{\ jmk|l}=R^i_{\ jml,k},\label{R8}\\
&&B^i_{\ jkl,m}=B^i_{\ jkm,l},
\end{eqnarray}
\emph{where `` $|$ " and ``, " \  denote the $h$- and $v$- covariant derivatives  with respect to the  Berwald connection, respectively.}
\end{lem}
\bigskip

\begin{thm}
(Sakaguchi) For any Finsler metric of scalar flag curvature, there is a tensor $D_{jkl}$ such that
\be
D^i_{\ jkl|m}y^m=D_{jkl}y^i.
\ee
This means that every Weyl metric is a generalized Douglas-Weyl metric.
\end{thm}
\begin{proof}
By definition, we have
\begin{equation}
D^i_{\ jkl}=B^i_{\ jkl}-\frac{2}{n+1}\{E_{jk}\delta^i_{\ l}+E_{kl}\delta^i_{\ j}+E_{lj}\delta^i_{\ k}+E_{jk,l}y^i\}.\label{GD2}
\end{equation}
Taking a horizontal derivation of (\ref{GD2}) yields
\begin{equation}
D^i_{\ jkl|m}y^m=B^i_{\ jkl|m}y^m-\frac{2}{n+1}\Big\{H_{jk}\delta^i_{\ l}+H_{kl}\delta^i_{\ j}+H_{lj}\delta^i_{\ k}+E_{jk,l|m}y^my^i\Big\},\label{D1}
\end{equation}
where $H_{ij}:=E_{ij|m}y^m$ (see \cite{NST2}). We have
\begin{eqnarray}
R^i_{\ jkl}\!\!\!\!&=&\!\!\!\!\ \frac{1}{3}\Big\{\frac{\partial^2 R^i_{\ k}}{\partial y^j \partial y^l}-\frac{\partial^2 R^i_{\ l}}{\partial y^j \partial y^k}\Big\}.\label{H5}
\end{eqnarray}
By assumption, $F$ is of scalar curvature ${\bf K}={\bf K}(x,y)$, which is equivalent to $R^i_{\ k} = {\bf K}F^2h^i_k$. Plugging it into (\ref{H5}) gives
\begin{eqnarray}\label{H20}
\nonumber R^i_{\ jkl}\!\!\!\!&=&\!\!\!\!\ \frac{1}{3}F^2\{{\bf K}_{,j,l}h^i_k-{\bf K}_{,j,k}h^i_l\}+{\bf K}_{,j}\{FF_{y^l}h^i_k-FF_{y^k}h^i_l\}\\ \nonumber
\!\!\!\!&&\!\!\!\!\ +\frac{1}{3} {\bf K}_{,k}\{2FF_{y^j}\delta^i_l-g_{jl}y^i-FF_{y^l}\delta^i_j\}+{\bf K}\{g_{jl}\delta^i_k-g_{jk}\delta^i_l\}\\
\!\!\!\!&&\!\!\!\!\ +\frac{1}{3} {\bf K}_{,l}\{2FF_{y^j}\delta^i_k-g_{jk}y^i-FF_{y^k}\delta^i_j\}.
\end{eqnarray}
Differentiating (\ref{H20}) with respect to $y^m$ gives a formula for $R^i_{\ jkl,m}$ expressed in terms of ${\bf K}$ and its derivatives. Contracting  (\ref{R8}) with $y^k$, we  obtain
\begin{eqnarray}
B^i_{\ jml|k}y^k =2{\bf K}C_{jlm}y^i \!\!\!\!&-&\!\!\!\!\! \frac{{\bf K}_{,j}}{3}\Big\{FF_{y^l}\delta^i_m+FF_{y^m}\delta^i_l-2g_{lm}y^i \Big \} \nonumber\\
\!\!\!\!&-&\!\!\!\!\! \frac{{\bf K}_{,l}}{3}\Big\{FF_{y^j}\delta^i_m+FF_{y^m}\delta^i_j-2g_{jm}y^i \Big \}\nonumber\\
\!\!\!\!&-&\!\!\!\!\! \frac{{\bf K}_{,m}}{3}\Big\{FF_{y^j}\delta^i_l+FF_{y^l}\delta^i_j-2g_{jl}y^i \Big \} \nonumber\\
\!\!\!\!&-&\!\!\!\!\! \frac{1}{3}F^2\Big\{{\bf K}_{,j,m}h^i_l+{\bf K}_{,j,l}h^i_m+{\bf K}_{,l,m}h^i_j\Big\},\label{H21}
\end{eqnarray}
where $C_{jlm}:=\frac{1}{4}[F^2]_{y^jy^ly^m}$ denotes the Cartan torsion of $F$. For (\ref{H21}), see (11.24) in \cite{Sh}. It follows from (\ref{H21}) that
\be
H_{jk}=-\frac{n+1}{6}\Big\{y_l{\bf K}_{,j}+y_j{\bf K}_{,l}+F^2{\bf K}_{,j,l}\Big\}.\label{H22}
\ee
By plugging (\ref{H21}) and  (\ref{H22}) in  (\ref{D1}) we get
\begin{eqnarray}
\nonumber D^i_{\ jkl|m}y^m=&&2{\bf K}C_{jlm}y^i +\frac{2}{3}\Big\{ {\bf K}_{,j}g_{kl}+{\bf K}_{,l}g_{kj}+{\bf K}_{,k}g_{jl}\Big\}y^i\\
&&+\frac{1}{3}\Big\{ y_{,j}{\bf K}_{,k,l}+  y_{,k}{\bf K}_{,j,l}+ y_{,l}{\bf K}_{,k,j}\Big\}y^i- \frac{2}{n+1}E_{jk,l|m}y^my^i.
\end{eqnarray}
Then every Weyl metric is a generalized Douglas-Weyl metric.
\end{proof}

%--------------------------------------------------------------------------------------------------------------------
\section{Proof of Theorem \ref{mainthm1}}
%--------------------------------------------------------------------------------------------------------------------
In this section, we are going to prove Theorem \ref{mainthm1}. Indeed, we characterize regular generalized Douglas-Weyl metric $(\alpha,\beta)$-metrics $F=\alpha\phi(s)$, $s=\beta/\alpha$, with vanishing S-curvature.

\bigskip

\noindent
{\bf Proof of Theorem \ref{mainthm1}.} Since $F$ has vanishing S-curvature then by (\ref{44}) we have $r_{ij}=0$ and $s_j=0$. Then (\ref{G1}) reduces to following
\begin{eqnarray}
G^i=G^i_{\alpha}+\alpha Q s^i_{\ 0}.\label{045}
\end{eqnarray}
By differential (\ref{045})  with respect to $y^j$, $y^l$ and $y^k$ we get
\begin{eqnarray}
\nonumber B^i_{\ jkl}=\!\!\!\!&&\!\!\!\! s^i_{\ l}\Big[Q\alpha_{jk}+Q_{k}\alpha_j+Q_{j}\alpha_k+\alpha Q_{jk}\Big]+ s^i_{\ j}\Big[Q\alpha_{lk}+Q_{k}\alpha_l+Q_{l}\alpha_k+\alpha Q_{lk}\Big]\\
\!\!\!\!&+&\!\!\!\! \nonumber s^i_{\ k}\Big[Q\alpha_{jl}+Q_{j}\alpha_l+Q_{l}\alpha_j+\alpha Q_{jl}\Big]+ s^i_{\ 0}\Big[\alpha_{jkl}Q+\alpha_{jk}Q_l+\alpha_{lk}Q_j+\alpha_{lj}Q_k\Big]\\
\!\!\!\!&+&\!\!\!\!  s^i_{\ 0}\Big[\alpha Q_{jkl}+\alpha_lQ_{jk}+\alpha_jQ_{lk}+\alpha_kQ_{jl}\Big].\label{046}
\end{eqnarray}
Contracting (\ref{046}) with $h^m_i$ implies that
\begin{eqnarray}
\nonumber h^m_iB^i_{\ jkl}=\!\!\!\!\!\!&&\!\!\!\!\!\!   s^m_{\ 0}\Big[\alpha_{jkl}Q+\alpha_{jk}Q_l+\alpha_{lk}Q_j+\alpha_{lj}Q_k+\alpha Q_{jkl}+\alpha_lQ_{jk}+\alpha_jQ_{lk}+\alpha_kQ_{jl}\Big]\\
\nonumber\!\!\!\!\!\!&+&\!\!\!\!\!\! \Big(s^m_{\ l}-F^{-2}s^0_{\ l}y^m\Big)\Big[Q\alpha_{jk}+Q_{k}\alpha_j+Q_{j}\alpha_k+\alpha Q_{jk}\Big]\\
\!\!\!\!\!\!&+&\!\!\!\!\!\! \nonumber \Big(s^m_{\ j}-F^{-2}s^0_{\ j}y^m\Big)\Big[Q\alpha_{lk}+Q_{k}\alpha_l+Q_{l}\alpha_k+\alpha Q_{lk}\Big]\\
\!\!\!\!\!\!&+&\!\!\!\!\!\!  \Big(s^m_{\ k}-F^{-2}s^0_{\ k}y^m\Big)\Big[Q\alpha_{jl}+Q_{j}\alpha_l+Q_{l}\alpha_j+\alpha Q_{jl}\Big].\label{047}
\end{eqnarray}
Taking a horizontal derivation of Douglas curvature  and contracting the result with $h^m_i$ implies that
\begin{equation}
h^m_iD^i_{\ jkl|s}y^s=h^m_iB^i_{\ jkl|s}y^s-\frac{2}{n+1}\Big\{H_{jk}h^m_l+H_{kl}h^m_j+H_{lj}h^m_k\Big\}.\label{GD3}
\end{equation}
By assumption ${\bf S}=0$ and then $H_{ij}=0$. Since $h^m_{\ i|s}=0$, then (\ref{GD3}) reduces to following
\be
\big(h^m_iB^i_{\ jkl}\big)_{|s}y^s=h^m_iB^i_{\ jkl|s}y^s=0.\label{GD4}
\ee
By (\ref{047}) and (\ref{GD4}), we have
\begin{eqnarray}
\nonumber h^m_iB^i_{\ jkl|s}y^s= \!\!\!\!\!&&\!\!\!\! s^m_{\ 0|0}\Big(\alpha_{jkl}Q+\alpha_{jk}Q_l+\alpha_{lk}Q_j+\alpha_{lj}Q_k+\alpha Q_{jkl}+\alpha_lQ_{jk}+\alpha_jQ_{lk}+\alpha_kQ_{jl}\Big)
\\
\nonumber \!\!\!\!\!&+&\!\!\!\! \Big(s^m_{\ l}-F^{-2}s^0_{\ l}y^m\Big)\Big(Q_{|0}\alpha_{jk}+Q_{k|0}\alpha_j+Q_{j|0}\alpha_k+\alpha Q_{jk|0}\Big)
\\
\nonumber \!\!\!\!\!&+&\!\!\!\! \Big(s^m_{\ j}-F^{-2}s^0_{\ j}y^m\Big)\Big(Q|_0\alpha_{lk}+Q_{k|0}\alpha_l+Q_{l|0}\alpha_k+\alpha Q_{lk|0}\Big)
\\
\nonumber \!\!\!\!\!&+&\!\!\!\! \Big(s^m_{\ k}-F^{-2}s^0_{\ k}y^m\Big)\Big(Q|_0\alpha_{jl}+Q_{j|0}\alpha_l+Q_{l|0}\alpha_j+\alpha Q_{jl|0}\Big)
\\
\nonumber \!\!\!\!\!&+&\!\!\!\! \Big(s^m_{\ l|0}-F^{-2}s^0_{\ l|0}y^m\Big)\Big(Q\alpha_{jk}+Q_{k}\alpha_j+Q_{j}\alpha_k+\alpha Q_{jk}\Big)
\\
\nonumber \!\!\!\!\!&+&\!\!\!\! \Big(s^m_{\ j|0}-F^{-2}s^0_{\ j|0}y^m\Big)\Big(Q\alpha_{lk}+Q_{k}\alpha_l+Q_{l}\alpha_k+\alpha Q_{lk}\Big)
\\
\nonumber \!\!\!\!\!&+&\!\!\!\!  \Big(s^m_{\ k|0}-F^{-2}s^0_{\ k|0}y^m \Big)\Big(Q\alpha_{jl}+Q_{j}\alpha_l+Q_{l}\alpha_j+\alpha Q_{jl}\Big)
\\
\nonumber\!\!\!\!\!&+&\!\!\!\! s^m_{\ 0}\Big(\alpha_{jkl}Q_{|0}+\alpha_{jk}Q_{l|0}+\alpha_{lk}Q_{j|0}+\alpha_{lj}Q_{k|0}+\alpha Q_{jkl|0}\Big)
\\
\!\!\!\!\!&+&\!\!\!\! s^m_{\ 0}\Big( \alpha_lQ_{jk|0}+\alpha_jQ_{lk|0}+\alpha_kQ_{jl|0}\Big)=0.\label{048}
\end{eqnarray}
Taking a horizontal derivation of  $y_ms^m_{\ 0}=0$ implies that $y_{m|0}s^m_{\ 0}+y_ms^m_{\ 0|0}=0$. Since $y_{m|0}=0$, then we get $y_ms^m_{\ 0|0}=0$. Therefore  by contracting (\ref{048}) with $y_m$ it follows that
\begin{eqnarray}
\nonumber  \Big(1-\alpha^2F^{-2} \Big)\Bigg[\!\!\!\!\!&&\!\!\!\!\!
s^0_{\ l}\Big(Q_{|0}\alpha_{jk}+Q_{k|0}\alpha_j+Q_{j|0}\alpha_k+\alpha Q_{jk|0}\Big)\\
\nonumber \!\!\!\!\!&+&\!\!\!\!\!  s^0_{\ j}\Big(Q|_0\alpha_{lk}+Q_{k|0}\alpha_l+Q_{l|0}\alpha_k+\alpha Q_{lk|0}\Big)\\
\nonumber\!\!\!\!\!&+&\!\!\!\!\!  s^0_{\  k}\Big(Q|_0\alpha_{jl}+Q_{j|0}\alpha_l+Q_{l|0}\alpha_j+\alpha Q_{jl|0}\Big)\\
\nonumber\!\!\!\!\!&+&\!\!\!\!\!  s^0_{\ l|0}\Big(Q\alpha_{jk}+Q_{k}\alpha_j+Q_{j}\alpha_k+\alpha Q_{jk}\Big)\\
\nonumber\!\!\!\!\!&+&\!\!\!\!\!  s^0_{\ j|0}\Big(Q\alpha_{lk}+Q_{k}\alpha_l+Q_{l}\alpha_k+\alpha Q_{lk}\Big)\\
\!\!\!\!\!&+&\!\!\!\!\!  s^0_{\ k|0}\Big(Q\alpha_{jl}+Q_{j}\alpha_l+Q_{l}\alpha_j+\alpha Q_{jl}\Big)
\Bigg]=0.\label{0.481}
\end{eqnarray}
Since $F$ is not Riemannian, then   $(1-F^{-2}\alpha^2)\neq 0$. Therefore   (\ref{0.481}) reduces to following
\begin{eqnarray}
\nonumber
\!\!\!\!\!&&\!\!\!\!\! s^0_{\ l}\Big(Q_{|0}\alpha_{jk}+Q_{k|0}\alpha_j+Q_{j|0}\alpha_k+\alpha Q_{jk|0}\Big)+s^0_{\ l|0}\Big(Q\alpha_{jk}+Q_{k}\alpha_j+Q_{j}\alpha_k+\alpha Q_{jk}\Big)\\
\nonumber \!\!\!\!\!&+&\!\!\!\!\!  s^0_{\ j}\Big(Q|_0\alpha_{lk}+Q_{k|0}\alpha_l+Q_{l|0}\alpha_k+\alpha Q_{lk|0}\Big)+s^0_{\ j|0}\Big(Q\alpha_{lk}+Q_{k}\alpha_l+Q_{l}\alpha_k+\alpha Q_{lk}\Big)\\
\!\!\!\!\!&+&\!\!\!\!\!  s^0_{\ k}\Big(Q|_0\alpha_{jl}+Q_{j|0}\alpha_l+Q_{l|0}\alpha_j+\alpha Q_{jl|0}\Big)+  s^0_{\ k|0}\Big(Q\alpha_{jl}+Q_{j}\alpha_l+Q_{l}\alpha_j+\alpha Q_{jl}\Big)
=0.\label{0.49}
\end{eqnarray}
Since $s_j=b_ms^m_{\ j}=0$, then we have
\begin{eqnarray}
0=(s_j)_{|0}=\big(b_ms^m_{\ j}\big)_{|0}=b_{m|0}s^m_{\ j}+b_ms^m_{\ j|0}=(r_{m0}+s_{m0})s^m_{\ j}+b_ms^m_{\ j|0}.\label{049}
\end{eqnarray}
Since  $r_{00}=0$, then (\ref{049}) reduces to following
\begin{eqnarray}
b_ms^m_{\ j|0}=-s_{m0}s^m_{\ j}.\label{050}
\end{eqnarray}
Multiplying  (\ref{050}) with $y^j$ yields
\be
b_ms^m_{\ 0|0}=-s_{m0}s^m_{\ 0}.\label{ee1}
\ee
Contracting (\ref{050}) with  $b^j$ implies that
\be
b^jb_ms^m_{\ j|0}=0.\label{ee2}
\ee
By considering (\ref{ee1}) and (\ref{ee2}) and multiplying (\ref{0.49}) with $b^jb^kb^l$ it follows that
\begin{eqnarray}
s^{m}_{\ 0}s_{m0}\Big(\alpha_2Q+2\alpha_1Q_1+\alpha Q_2\Big)=0,\label{054}
 \end{eqnarray}
where
\begin{eqnarray*}
\alpha_1: = b^i\alpha_{y^i}, \ \ \ \alpha_2:=b^ib^j\alpha_{y^iy^j}, \ \ \ \ Q_1: = b^iQ_{y^i}, \  \ Q_2:=b^ib^jQ_{y^iy^j}.
\end{eqnarray*}
By (\ref{054}),  we have $Q\alpha_2+2\alpha_1Q_1+\alpha Q_2=0$ or $s^{m}_{\ 0}s_{m0}=0$. Suppose that, the first case holds
\be
Q\alpha_2+2\alpha_1Q_1+\alpha Q_2=0.\label{eq}
\ee
 By definition, we have the following
\begin{eqnarray*}
&&\alpha_{y^i}=\alpha^{-1}y_i, \ \ \ \alpha_{y^jy^k}=\alpha^{-3}A_{jk}, \ \ \ \ \alpha_{y^jy^ky^l}=-\alpha^{-5}[A_{jk}y_l+A_{jl}y_k+A_{lk}y_j],
\end{eqnarray*}
where $A_{jk}:=\alpha^2a_{jk}-y_jy_k$. So we get
\begin{eqnarray}
&& \alpha_1=s,\label{a1}\\
&& \alpha_2=(b^2-s^2)\alpha^{-1},\label{a2}\\
&&\alpha_3=-3s(b^2-s^2)\alpha^{-2},\label{003}\\
&&Q_1= Q'\frac{(b^2-s^2)}{\alpha},\label{a3}\\
&&Q_2=\frac{(b^2-s^2)\big[(b^2-s^2)Q''-3sQ'\big]}{\alpha^2},\label{005}\\
&& Q_3= \frac{3(b^2-s^2)\big[(b^2-s^2)^2Q'''-3s(b^2-s^2)Q''-(b^2-5s^2)Q'\big]}{\alpha^3}.\label{006}
\end{eqnarray}
By plugging (\ref{a1}), (\ref{a2}), (\ref{003}), (\ref{a3}),  (\ref{005}) and (\ref{006}) into (\ref{eq}),  we  get
\begin{eqnarray*}
(b^2-s^2)Q''+(Q-s Q')=0,
\end{eqnarray*}
which implies that
\begin{eqnarray}
Q=ks+q\sqrt{b^2-s^2},\label{eq2}
\end{eqnarray}
where $k$ and $q$ were constants. By (\ref{eq2}), it follows that
\begin{eqnarray}
\phi=c\exp\Bigg[\int_0^s{\frac{kt+q\sqrt{b^2-t^2}}{1+kt^2+qt\sqrt{b^2-t^2}}dt}\Bigg],\label{uni}
\end{eqnarray}
where $c$ is a real constant. But, (\ref{uni}) is an almost regular $(\alpha, \beta)$-metric (for more details, see \cite{ShC}). Consequently,  we conclude that $s^{m}_{\ 0}s_{m0}=0$ which implies that $\beta$ is closed. Therefore by  (\ref{046}), $F$  reduces to a Berwald metric.
\qed

\begin{rem}\label{rem}
In the Theorem \ref{mainthm1}, if we suppose  that $F=\alpha\phi(s)$ is almost regular metric and ${\bf B}\neq 0$ then it reduces to (\ref{uni}), where  $c>0$, $q>0$ and $k$ are real constants. Since ${\bf S}=0$, then $F$ can not be a Douglas metric. By Cheng' Theorem (Theorem 4 in \cite{Cheng}), if $F$ is a Weyl metric then it reduces to a Berwald metric. Thus $F$ is not a Weyl metric. By  Theorem 1.2 in \cite{ShC}, since $s_{ij}\neq 0$ then $F$ is not a Landsberg metric, also.
\end{rem}

\bigskip
\noindent
{\bf Proof of Corollary  \ref{cor1}:}  In \cite{Cui}, Cui shows that for the Matsumoto metric $F=\frac{\alpha^2}{\alpha-\beta}$ and the special  $(\alpha, \beta)$-metric $F=\alpha+\epsilon\beta+\kappa(\beta^2/\alpha)$ $(\kappa\neq 0)$,  ${\bf E}= \frac{n+1}{2} c F^{-1}{\bf h}$ if and only if ${\bf S}=0$. Then by Theorem \ref{mainthm1}, we get the proof.
\qed
\bigskip

In \cite{TSP}, it is proved that every R-quadratic Finsler metric is a generalized Douglas-Weyl metric.   Thus  by Theorem \ref{mainthm1}, we get the following.

\begin{cor}
Let $F=\alpha\phi(s)$, $s=\beta/\alpha$, be a regular  $(\alpha,\beta)$-metric on a manifold $M$ with vanishing S-curvature. Suppose that $F$ is not a Finsler metric of Randers type. Then  $F$  R-quadratic  if and only if it is a Berwald metric.
\end{cor}

%-----------------------------------------------------------------------------------

\bigskip

\noindent
Akbar Tayebi and Hassan Sadeghi\\
Department of Mathematics, Faculty  of Science\\
University of Qom \\
Qom. Iran\\
Email:\ akbar.tayebi@gmail.com\\
Email:\ sadeghihassan64@gmail.com

\begin{thebibliography}{MaHo}
%---------------------------------------------------------------
\bibitem{BCZ} S.  B\'{a}cs\'{o}, X. Cheng and Z. Shen, {\it  Curvature properties of $(\alpha, \beta)$-metrics}, Adv. Stud. Pure. Math. {\bf 48}(2007), 73-110.
%--------------------------------------------------------------
 \bibitem{BM} S.  B\'{a}cs\'{o} and M. Matsumoto, {\it On Finsler spaces of Douglas type, A generalization of notion of Berwald space}, Publ. Math. Debrecen. \textbf{51}(1997), 385-406.
%----------------------------------------------------------------
\bibitem{BP} S. B\'{a}cs\'{o} and  I. Papp, {\it A note on a generalized  Douglas space}, Periodica. Math. Hungarica. \textbf{48}(2004), 181-184.
%--------------------------------------------------------------
%\bibitem{BS} D. Bao and Z. Shen, {\it Finsler metrics of constant positive curvature on the Lie group $S^3$}, J. London Math. Soc.  \textbf{66} (2002), 453-467.
%----------------------------------------------------------------------------------------------------
\bibitem{Cheng} X. Cheng, {\it On $(\alpha,\beta)$-metrics of scalar flag curvature with constant S-curvature}, Acta Mathematica Sinica, English Series, {\bf 26}(9) (2010), 1701-1708.
%----------------------------------------------------------------------------------------------------
\bibitem{CL} X. Cheng and C. Lu, {\it Two kinds of weak Berwald metrics of scalar flag curvature}, J. Math. Research. Exposition. {\bf 29}(4) (2009), 607-614.
%----------------------------------------------------------------------------------------------------
\bibitem{ChSh3} X. Cheng and Z. Shen, {\it A class of Finsler metrics with isotropic S-curvature}, Israel J. Math. {\bf 169}(2009), 317-340.
%--------------------------------------------------------------
\bibitem{Cui} N. Cui, {\it On the S-curvature of some $(\alpha,\beta)$-metrics}, Acta. Math. Scientia, Series: A. {\bf 26}(7) (2006), 1047-1056.
 %-------------------------------------------------------------
 \bibitem{Mat} M. Matsumoto,  {\it Theory of Finsler spaces with $(\alpha, \beta)$-metric}, Rep. Math. Phys. {\bf 31}(1992), 43-84.
%--------------------------------------------------------------
\bibitem{MY}  X. Mo and C. Yu, {\it On the Ricci curvature of a Randers metrics of isotropic S-curvature}, Acta Mathematica Sinica, English Series, {\bf 24}(2008), 991-996.
%--------------------------------------------------------------
\bibitem{NST1} B. Najafi, Z. Shen and A. Tayebi,  {\it On a projective class   of Finsler metrics}, Publ. Math. Debrecen. {\bf 70}(2007), 211-219.
%--------------------------------------------------------------
\bibitem{NST2} B. Najafi, Z. Shen and A. Tayebi, {\it Finsler metrics of scalar flag curvature with special non-Riemannian curvature properties}, Geometriae Dedicata, {\bf 131}(2008), 87-97.
%--------------------------------------------------------------
\bibitem{Sak} T. Sakaguchi, {\it On Finsler spaces of scalar curvature}, Tensor N. S. {\bf 38}(1982), 211-219.
%-----------------------------------------------------------------------------------
\bibitem{Shvol} Z. Shen, {\it  Volume comparison and its applications in Riemann-Finsler geometry}, Advances. Math. {\bf 128} (1997), 306-328.
%--------------------------------------------------------------
\bibitem{Sh} Z. Shen, {\it Differential Geometry of Spray and Finsler Spaces}, Kluwer Academic Publishers,  Dordrecht,  2001.
%-----------------------------------------------------------------------------------
\bibitem{ShZ} Z. Shen, {\it Finsler manifolds with nonpositive flag curvature and constant S-curvature},  Math. Z., {\bf 249}(3) (2005), 625-639.
%--------------------------------------------------------------
\bibitem{ShC} Z. Shen, {\it On a class of Landsberg metrics in Finsler geometry}, Canadian. J. Math. {\bf 61}(6) (2009), 1357-1374.
%-----------------------------------------------------------------------------------
\bibitem{SX} Z. Shen and H. Xing, {\it On Randers metrics of isotropic S-curvature}, Acta Mathematica Sinica, English Series, {\bf 24}(2008), 789-796.
%-----------------------------------------------------------------------------------
%\bibitem{Sh11} Z. Shen, {\it Two-dimensional Finsler metrics with constant flag curvature}, Manuscripta. Math. {\bf 109}(3) (2002), 349-366.
%-----------------------------------------------------------------------------------
%\bibitem{Sh12} Z. Shen, {\it Finsler metrics with ${\bf K}=0$ and ${\bf S}=0$}, Canadian. J. Math. {\bf 55}(1) (2003), 112-132.
%--------------------------------------------------------------
%\bibitem{TP1} A. Tayebi and E. Peyghan, {\it On a subclass of generalized Douglas-Weyl metrics}, J. Cont. Math. Analysis.   {\bf 47}(2) (2012), 69-80.
%--------------------------------------------------------------
%\bibitem{TP2} A. Tayebi and E. Peyghan, {\it On special generalized Douglas-Weyl metrics}, J. Sci. Islamic. Republic.  Iran.  {\bf 23}(2) (2012), 179-184.
%--------------------------------------------------------------
\bibitem{TR} A. Tayebi and M. Rafie. Rad, {\it S-curvature of isotropic Berwald metrics}, Sci. China. Series A: Mathematics. {\bf 51}(2008), 2198-2204.
%--------------------------------------------------------------
\bibitem{TSP} A. Tayebi, H. Sadeghi and E. Peyghan, {\it On generalized Douglas-Weyl spaces}, Bull. Malays. Math. Sci. Soc. (2) {\bf 36}(3) (2013), 587-594.
%--------------------------------------------------------------

\end{thebibliography}
\end{document}